\documentclass{article}

\usepackage[centertags]{amsmath}
\usepackage{amssymb}
\usepackage{mathptmx}
\usepackage[colorlinks=true,linkcolor=blue,citecolor=blue]{hyperref}
\usepackage{subfig}
\usepackage{amsthm}

\usepackage{adjustbox}

\usepackage{tikz}
\usetikzlibrary{decorations.pathmorphing}
\usepackage{graphicx, xcolor, soul}

\usepackage[colorinlistoftodos,prependcaption,textsize=tiny]{todonotes}

\usepackage{url}
\usepackage[T1]{fontenc}

\usepackage{tikz}
\usetikzlibrary{calc}

\title{A Note on Cores and Quasi Relative Interiors in Partially Finite Convex Programming}

\author{Scott B. Lindstrom\\Hong Kong Polytechnic University}

\date{\today}

\def\RR{\hbox{$\mathbb R$}}

\def\d{\hbox{\rm ds}}
\def\W{\hbox{$\mathcal W$}}
\def\boldx{\hbox{$\mathbf{x}$}}

\def\inte{\hbox{\rm int}}

\newcommand{\qri}{\ensuremath{\operatorname{qri}}}
\newcommand{\reli}{\ensuremath{\operatorname{ri}}}
\newcommand{\cone}{\ensuremath{\operatorname{cone}}}
\newcommand{\cl}{\ensuremath{\operatorname{cl}}}
\newcommand{\ent}{\ensuremath{\operatorname{ent}}}

\providecommand{\core}{\operatorname{core}}
\newcommand{\dom}{\ensuremath{\rm dom}}
\newcommand{\aff}{\ensuremath{\rm aff}}

\newtheorem{theorem}{Theorem}[section]

\newtheorem{lemma}[theorem]{Lemma}
\newtheorem{proposition}[theorem]{Proposition}
\theoremstyle{definition}
\newtheorem{definition}[theorem]{Definition}
\newtheorem{example}{\it Example}
\newtheorem{remark}[theorem]{Remark}

\makeatletter
\newenvironment{varsubequations}[1]
{%
	\addtocounter{equation}{-1}%
	\begin{subequations}
		\renewcommand{\theparentequation}{#1}%
		\def\@currentlabel{#1}%
	}
	{%
	\end{subequations}\ignorespacesafterend
}
\makeatother

\begin{document}
	
	\maketitle
	
	\begin{abstract}
		The problem of minimizing an entropy functional subject to linear constraints is a useful example of partially finite convex programming. In the 1990s, Borwein and Lewis provided broad and easy-to-verify conditions that guarantee strong duality for such problems. Their approach is to construct a function in the quasi-relative interior of the relevant infinite-dimensional set, which assures the existence of a point in the core of the relevant finite-dimensional set. We revisit this problem, and provide an alternative proof by directly appealing to the definition of the core, rather than by relying on any properties of the quasi-relative interior. Our approach admits a minor relaxation of the linear independence requirements in Borwein and Lewis' framework, which allows us to work with certain piecewise-defined moment functions precluded by their conditions. We provide such a computed example that illustrates how this relaxation may be used to tame observed Gibbs phenomenon when the underlying data is discontinuous. The relaxation illustrates the understanding we may gain by tackling partially-finite problems from both the finite-dimensional and infinite-dimensional sides. The comparison of these two approaches is informative, as both proofs are constructive.
	\end{abstract}

	\section{Introduction}
	
	In the 1990s, Borwein and Lewis introduced the \emph{quasi-relative interior} ($\qri$) of a set $C$ \cite{borwein1988partiallyparts1and2}, defined by
	\begin{equation}
	\qri C = \{x \in C\; \text{such\;that}\; \cl \cone (C-x) \;\text{is\;a\;subspace} \}.
	\end{equation}
	The notion of the quasi-relative interior specifies to that of the relative interior ($\reli$) on $\RR^n$, while $\qri C \supset \reli C$ holds in a Hausdorff topological vector space of arbitrary dimension. Borwein and Lewis demonstrated the value of the quasi-relative interior in the context of the partially finite convex programming problem of constrained entropy optimization \cite{borwein1993partially,borwein1992entropy}, which we now recall. 
	
	Let $f$ be a proper, closed convex function. We define its associated entropy functional by
	\begin{align*}
	I_f : L^1([0,\tau]) &\rightarrow \mathbb{R} \\
	\text{by} \quad I_f : x &\mapsto \int_{0}^{\tau} f(x(s)) ds.
	\end{align*}
	Some commonly employed entropy functions and their conjugates are listed in \cite[Table~1]{borwein1992entropy}, some of which we recall in Table~\ref{tab:entropies}.
	\begin{table}\small
		\begin{tabular}{l l l}
			name & definition & conjugate \\ \hline
			$L^2$ Norm & $\frac{1}{2}u^2$& $\frac12v^2$ \\
			{\scriptsize (negative)} Boltzmann--Shannon & $u\log(u)$ & $\exp(v-1)$ \\
			{\scriptsize (negative translated)} Boltzmann--Shannon & $u\log(u)-u =:\ent(u)$ & $\exp(v)$ \\
			Burg & $-\log(u)$ & $-1-\log(-v)$ \\
			Cosh & $\cosh(u)$ & ${\rm arcsinh}(v)-\sqrt{1+v^2}$ \\
			Fermi Dirac & $u\log(u)+(1-u)\log(1-u)$ & $\log(1+\exp(v))$ \\
		\end{tabular}
	\caption{A selection of entropies that includes those from \cite{borwein1992entropy}.}\label{tab:entropies}
	\end{table}
	Borwein and Lindstrom have recently illuminated the role that real principal branch of the Lambert $\W$ function plays in the convex conjugates for weighted sums of entropies \cite{BL2016}; such sums admit new entropies. The entropy optimization problem is to minimize $I_f$ subject to the $n$ continuous linear constraints of the form
	\begin{equation*}
	\langle a_k,x \rangle = \int_0^\tau a_k(s)x(s) \d = b_k, \quad k=1,\dots,n.
	\end{equation*}
	We may describe these linear equality constraints concisely in terms of the linear operator $A$ as follows:
	\begin{align}
	A:L^1([0,\tau]) &\rightarrow \mathbb{R}^n \nonumber \\
	\text{by}\quad A: x & \mapsto \left(\int_0^\tau a_1(s)x(s) \d , \dots , \int_0^\tau a_n(s)x(s) \right)=b\nonumber \\
	\text{where}\quad b:= A\rho, \label{d:A}
	\end{align}
	where $\rho, a_k \in L^\infty ([0,\tau])$ and $\rho$ is a given function used to generate the linear equality constraint vector $b$. When $f^*$ is smooth and everywhere finite on $\RR$, the problem
	\begin{equation}\label{Iprimalproblem}
	\underset{x\in L^1}{\inf}\{I_f(x) \;|\; Ax=b \}
	\end{equation}
	reduces to solving the finite nonlinear equation
	\begin{equation}\label{optimalmultipliers}
	\int_0^\tau (f^*)' \left(\sum_{j=1}^n \mu_j a_j(s) \right) a_k(s) \d = b_k \quad (1\leq k \leq n).
	\end{equation}
	We will recall the reasons for this in Section~\ref{conjugateduality}, and additional information about primal attainment may be found in \cite{BL2}. 
	
	The structure of the chosen entropy function $f$ may impose hard constraints on the solution. For example, when $f=\ent$ is the negative, translated Boltzmann--Shannon entropy, the feasible region consists of nonnegative functions. In a follow-up to \cite{BL2016}, Bauschke and Lindstrom computed entropy optimization problems using \emph{proximal} averages of entropy functions \cite{BL2018}. In particular, they showed how proximal averages allow the practitioner to exercise flexibility with the barriers that entropies afford, in order to obtain a larger feasible region.
	
	\subsection{Outline}
	
	The remainder of this paper is outlined as follows. In Section~\ref{conjugateduality}, we recall the circumstances under which \eqref{Iprimalproblem} is equivalent to \eqref{optimalmultipliers}, excepting a detailed discussion of how to verify the strong duality. In Section~\ref{s:newresults}, we provide our new approach to developing easy-to-verify sufficient conditions for strong duality, by appealing directly to the definition of the core. In Section~\ref{s:existingtheory}, we briefly recall the classical approach that relies on finding a function in a quasi-relative interior. In Section~\ref{s:example}, we provide a computed example that illustrates the minor relaxation we have obtained. We conclude in Section~\ref{s:conclusion}.

	\section{On the equivalence of the dual problem and the Lagrange multiplier problem}\label{conjugateduality}\index{conjugate duality}
	
	This discussion closely follows the one found in \cite{ScottPhD}, which is an augmented version of those in  \cite{BL2018,BL2016}. To see why solving \eqref{Iprimalproblem} reduces to solving \eqref{optimalmultipliers}, we will first recall several results.

		\begin{remark}\label{rem:partialdomain}
		Let $f:X \rightarrow \left]-\infty,+\infty \right[$ be proper. Then $f^*:X \rightarrow \left]-\infty,+\infty \right]$ is proper. Let $x,u \in X$. Then it holds that
		$$
		u \in \partial f^*(x) \iff f(u)+f^*(x)= \langle x,u \rangle \iff x \in \partial f(u).
		$$
		For details, see \cite[Proposition 16.9]{BC}. Thus it holds that
		$$
		{\rm ran} (\partial f^*) \subset {\rm dom}(\partial f) \subset {\rm dom}(f).
		$$
		Consequently, we have that:
		$$
		(\forall x \in X)\quad f^*(x) \in {\rm dom}(f) = \left [0,\infty \right[.
		$$
	\end{remark}

	We next recall a theorem (Theorem~\ref{thm:borweinzhuduality}) from \cite[Theorem 4.7.1]{BZ} that provides sufficient conditions for strong duality to hold. For the present exposition, we are particularly interested in how to verify those sufficient conditions. We will introduce an easy-to-check criterion in Section~\ref{s:newresults}, and then we will compare it with the existing theory in Section~\ref{s:existingtheory}. Our process uses the definition of the \emph{core}, which we now introduce.
	\begin{definition}[Core of a set]\label{def:core}
	The \emph{core} of a set $C\subset \RR^n$, denoted by $\core C$, is the set of points $x \in C$ such that for any direction vector $\eta \in \RR^n$, $x+t\eta \in C$ for all $t>0$ sufficiently small.
	\end{definition}
	
	\begin{theorem}[{\cite[Theorem 4.7.1]{BZ}}]\label{thm:borweinzhuduality}
		Let $X$ be a Banach space and $F: X \rightarrow \mathbb{R}\cup \{+\infty\}$ be a lower semicontinuous convex function. Let $A: X \rightarrow \mathbb{R}^n$ be a linear operator, and $b \in {\; \rm core}(A {\; \rm dom}F)$. Then
		\begin{equation*}
			\inf_{x\in X}\left\{ F(x)\; |\; Ax=b \right\} = \max_{\varphi \in \mathbb{R}^N}\left\{\langle \varphi, b\rangle -F^*(A^T \varphi) \right\},
		\end{equation*}
		where $A^T$ denotes the adjoint map that satisfies
		\begin{align*}
			A^T: \mathbb{R}^n \rightarrow X\nonumber \quad \text{by} \quad \langle Au,\varphi \rangle_{\mathbb{R}^n} = \langle u,A^T \varphi \rangle_X.
		\end{align*}
	\end{theorem}

	Our new contribution in Section~\ref{s:newresults} is an easy-to-verify sufficient condition to apply Theorem~\ref{thm:borweinzhuduality}. For now, we suppose that the problem from \eqref{Iprimalproblem} satisfies the sufficient conditions to apply Theorem~\ref{thm:borweinzhuduality}, and so we may reformulate \eqref{Iprimalproblem} as
	\begin{equation}\label{Idualproblem}
		\inf_{x\in L^2}\left\{ I_f(x) \;|\; Ax=b \right\} = \max_{\varphi \in \mathbb{R}^N}\left\{\langle \varphi, b\rangle -(I_f)^*(A^T \varphi) \right\},
	\end{equation}
	where $A^T: \mathbb{R}^n \rightarrow L^1$ is the adjoint map satisfying the equality
	\begin{align}
		\langle Au,\varphi \rangle_{\mathbb{R}^n} = \langle u,A^T \varphi \rangle_{L^2}. \label{adjoint}
	\end{align}
	
	To further simplify this, we use the following result from \cite[Theorem 6.3.4]{BV}.
	\begin{proposition}\label{prop:Ifduality}
		If	 $I_f$ is defined as above and $f:\mathbb{R}\rightarrow \left]-\infty,\infty \right]$ is convex, proper, closed, we have
		\begin{equation*}
			(I_f)^*=I_{f^*}.
		\end{equation*}
	\end{proposition}
	
	Proposition~\ref{prop:Ifduality} allows us to express the dual problem from the right hand side of \eqref{Idualproblem} explicitly as follows:
	\begin{equation}\label{Idualproblem2}
		\max_{\varphi \in \mathbb{R}^N}\left\{\langle \varphi, b\rangle - I_{f^*}(A^T \varphi) \right\}=\max_{\varphi \in \mathbb{R}^N}\left\{\langle \varphi, b\rangle -\int_0^\tau (f^* \circ A^T \varphi)(s){\; \rm d}s \right\}.
	\end{equation}
	We may express the adjoint map $A^T$ from \eqref{adjoint} explicitly as follows:
	\begin{align}
		\sum_{k=1}^{n}\left(\varphi_k\int_{0}^{\tau}a_k(s)u(s){\; \rm d}s\right)&=\langle Au,\varphi \rangle_{\mathbb{R}^n}\nonumber \\
		&=\langle u,A^T \varphi \rangle_{L^2}=\int_{0}^{\tau}(A^T\varphi)(s)u(s){\; \rm d}s. \label{e:adjoint}
	\end{align}
	From \eqref{e:adjoint}, it is clear why we may explicitly describe $A^T \varphi$ by:
	\begin{equation*}A^T\varphi = \sum_{k=1}^{n}\varphi_k a_k(s).\end{equation*}
	Thus solving \eqref{Idualproblem2} reduces to finding $\varphi \in \mathbb{R}^n$ that maximizes
	\begin{equation}\label{wecansubdifferentiate}
		\sum_{k=1}^{n}\varphi_k b_k - \int_{0}^{\tau}f^*\left(\sum_{k=1}^{n}\varphi_k a_k(s) \right){\; \rm d}s.
	\end{equation}
	We can subdifferentiate \eqref{wecansubdifferentiate} and recover its maximum by finding the values of $\varphi_1,\dots,\varphi_n$ for which the subdifferential of \eqref{wecansubdifferentiate} with respect to $\varphi$ is zero. For that purpose, we recall the Fenchel--Young \emph{inequality} of a convex function $f$,
	\begin{equation}\label{fenchelyounginequality}
	(\forall x, y \in X)\quad 0 \leq f(x)+f^*(y)-\langle x, y \rangle,
	\end{equation}
	which follows from the definition of the conjugate, and we also recall the Fenchel--Young \emph{equality} in the following Lemma.
	\begin{lemma}\label{subgradientlemma}
		For a convex function $f$, $y\in \partial f(x)$ if and only if
		\begin{equation}\label{fenchelyoungequality}
		0=f(x)+f^*(y)-\langle y,x \rangle.
		\end{equation}
	\end{lemma}
	\begin{proof}
		From the definition of the conjugate, \eqref{fenchelyoungequality} is simply
		\begin{equation*}
		0=-f(x)+\langle y,x \rangle -\sup_{u\in X}\{\langle y,u \rangle - f(u)\}.
		\end{equation*}
		This is equivalent to
		\begin{equation*}
		f(x)-\langle y,x \rangle = -\sup_{u\in X}\{\langle y,u \rangle - f(u)\}= \inf_{u\in X}\{f(u)-\langle y,u \rangle \},
		\end{equation*}
		which is equivalent to
		\begin{equation*}
		\langle y,u-x \rangle \leq f(u)-f(x)\text{ for all }u\in X.
		\end{equation*}
		This is equivalent to $y \in \partial f(x)$.
	\end{proof}
	Because $(I_f )^*=I_{f^*}$, we have that 
	\begin{equation*}
		(y\in \partial I_f(x))\;\; \iff \;\; 0 = I_f (x)+I_{f^*}(y) -\langle y,x \rangle = \int{ f (x(s))+f^*(y(s))  - x(s)y(s)} {\; \rm d}s.
	\end{equation*}
	By the Fenchel--Young \emph{equality} \eqref{fenchelyounginequality}, the integrand on the right is nonnegative, and so must be zero almost everywhere. Using Lemma~\ref{subgradientlemma}, we have that
	\begin{equation*}
		f (x(s))+f^*(y(s))  - x(s)y(s) = 0 \quad \iff y(s) \in \partial f (x(s)).
	\end{equation*}
	Thus we subdifferentiate with respect to each $\varphi_k$ in \eqref{wecansubdifferentiate} and set the subdifferential equal to zero, which yields the $n$ equations:
	\begin{equation*}
		0= b_k -  \int_{0}^{\tau}(f^*)'\left(\sum_{j=1}^{n}\varphi_j a_j(s) \right)a_k(s){\; \rm d}s \quad \text{for}\;\;k=1,\dots,n.
	\end{equation*}
	Thus we have reduced solving \eqref{Iprimalproblem} to the challenge of solving Equation~\eqref{optimalmultipliers}. When $\mu_1,\dots,\mu_n$ are the optimal multipliers in \eqref{optimalmultipliers}, the solution $x$ to the primal problem \eqref{Iprimalproblem} is 
	\begin{equation}\label{attainment}
	x(s) = \left(f^*\right)' \left(\sum_{j=1}^n \varphi_j a_j(s) \right).
	\end{equation}
	Sufficient conditions for attainment of this primal solution are given in \cite{BL2}. Now we will introduce our process for verifying sufficient conditions for strong duality.
	
	\section{Approach using the core}\label{s:newresults}\index{strong duality (sufficient conditions for entropy minimization)}
	
	This section is largely based on previously unpublished material from \cite{ScottPhD}. Our goal is to provide easy-to-check sufficient conditions to apply Theorem~\ref{thm:borweinzhuduality} with $F:= I_f$. 
	
	First, to see that $I_f$ is convex, notice that, $(\forall x,y \in L^1)\;(\forall \lambda \in \left[0,1\right])$,
	\begin{align*}
	I_f(\lambda x + (1-\lambda)y) &= \int_0^\tau f(\lambda x(s)+(1-\lambda)y(s)) ds\\
	&\leq \int_0^\tau \lambda f(x(s))+(1-\lambda)f(y(s)) ds\\
	&=\lambda \int_0^\tau f(x(s))ds + (1-\lambda) \int_0^\tau f(y(s)) ds,
	\end{align*}
	where the inequality follows from the convexity of $f$. Having shown the convexity of $I_f$, we need only a convenient way to verify that $b \in \core (A \dom I_f)$. The following description of the domain of $I_f$ is from \cite{ScottPhD}.
	
	 \begin{lemma}\label{lem:S}
		Let $[\alpha,\beta] \subset \dom f$ with $f$ convex. Let
		\begin{equation*}
		\mathbf{S}_f := \left \{x \in L^p([0,\tau]) \quad \text{such that} \quad \; \; x(s) \in [\alpha,\beta]\; s-a.e.   \right \}.
		\end{equation*}
		Then $\mathbf{S}_f \subset \dom I_f$.
	\end{lemma}
	\begin{proof}
		By definition,
		\begin{align*}
		\dom I_f &= \left \{x \in L^1([0,\tau])\quad \text{such that} \quad \int_0^\tau f(x(s)) ds  < \infty  \right \}.
		\end{align*}
		The conditions $x \in \mathbf{S}_f$, $\left[\alpha, \beta \right] \subset \dom f$, and $f$ convex together ensure that $f(x(s)) \leq \max\{f(\alpha),f(\beta)\} < \infty$ holds $s$-a.e. Thus 
		$$
		\int_0^\tau f(x(s)) ds \leq \int_0^\tau \max\{f(\alpha),f(\beta)\} ds \leq \tau \max\{f(\alpha),f(\beta)\} < \infty,
		$$
		and so $x \in \dom I_f$.
	\end{proof}
	The next example problem was used to showcase the role of the Lambert $\W$ function in entropy optimization in \cite{BL2018,BL2016}, and it illustrates Lemma~\ref{lem:S}. 
	\begin{example}\label{ex:pulse}
		Let $f$ be the Boltzmann--Shannon entropy or the $\lambda \in \left]0,1 \right[$-weighted average of the entropy and the $L^2$ norm: $u \mapsto \lambda u^2 + (1-\lambda)u\log(u)$. In both cases, $\left[0,\infty \right[ =  \dom f$. Now we define the pulse:
		\begin{equation}\label{def:pulse}
		\boldx: [0,1] \rightarrow \{0,1\}: s \mapsto \begin{cases}
		1 & \text{if}\; s \in [0,1/2]\\
		0 & \text{otherwise.}
		\end{cases}
		\end{equation}
		In this case, $\boldx \in \mathbf{S}_f$ with $\alpha=0$ and $\beta=1$ respectively.	
	\end{example}
	
	We will use the definition of the core in order to develop easy-to-verify conditions for strong duality in Theorem~\ref{thm:main}. The following Lemma will allow us to choose functions $y_1,\dots,y_n$ that explicitly tie our functions $a_1,\dots,a_n$ with a general direction vector $\eta \in \RR^n$ through the operator $A$. 
	
	Before we begin, we should note why we have chosen to work with the definition of the core instead of the interior. Since $I_f$ is convex, its domain is convex. Since $A \dom I_f$ is then a convex, finite-dimensional set, we have $\inte A \dom I_f = \core A \dom I_f$ (\cite[Theorem~4.1.4]{BL}). However, we choose to speak in terms of the core, because our proofs will explicitly use direction vectors rather than open balls. This choice of definition allows us to find a value $t>0$ that \emph{depends on the direction vector} $\eta$, rather than seeking to find \emph{one} value $t>0$ so that the inclusion holds for \emph{all} direction vectors as would be required by the definition of the interior.
	
	\begin{lemma}\label{lem:linearindependence}
		Let $M\subset L^p(\left[\zeta_1,\zeta_2\right])$ with $p\geq2$ be a subspace of dimension $n$ with basis $\{a_1,\dots,a_n \}$, and let $\eta \in \RR^n$. Let $\langle \cdot, \cdot \rangle$ denote the inner product on $L^2(\left[\zeta_1,\zeta_2 \right])$. Then, for any $k \in \{1,\dots,n\}$, we may find $y_k \in M$ such that
		\begin{align*}
		\langle y_k,a_k \rangle &= \eta_k\\
		\text{and}\quad \langle y_k,a_j \rangle &= 0 \quad \text{for}\; j \neq k.
		\end{align*} 
	\end{lemma}
	\begin{proof}
		The case where $n=1$ is obvious, so let $n \geq 2$. If $\eta_k = 0$, then let $y_k=0$ and the proof is complete. Otherwise, without loss of generality, let $k=1$. Let $v \in M \cap \left({\rm span} \{a_2,\dots,a_n\}\right)^\perp \setminus \{0\}$. Then $\{v,a_2,\dots,a_n \}$ is a basis for $M$, and so we may write
		\begin{equation}\label{asandvs}
		a_1 = \lambda_1 v + \lambda_2 a_2 + \dots + \lambda_n a_n
		\end{equation}
		for some $\lambda \in \RR^n$. By the linear independence of $\{a_1,\dots,a_n\}$, \eqref{asandvs} implies that 
		\begin{equation}\label{itsnotzero}
		\lambda_1 \neq 0. 
		\end{equation}
		Thus we have that
		\begin{align}
		\langle a_1,v \rangle &= \langle \lambda_1 v + \lambda_2 a_2 + \dots + \lambda_n a_n , v \rangle \nonumber\\
		&=\langle \lambda_1 v,v\rangle  + \langle \lambda_2 a_2 , v \rangle + \dots + \langle \lambda_n a_n ,v \rangle \nonumber \\
		&=\langle \lambda_1 v,v\rangle + 0 + \dots + 0 \nonumber\\
		&= \lambda_1 \|v\|_{\langle \cdot,\cdot \rangle}^2.
		\end{align}
		Since we have $\lambda_1 \neq 0$ and $v \neq 0$, we have $\lambda_1 \|v\|_{\langle, \cdot, \cdot, \rangle}^2 \neq 0$. Using the fact that $v \in M \subset L^p(\left[\zeta_1,\zeta_2\right]) \subset L^2(\left[\zeta_1,\zeta_2\right])$, we have that $\|v\|_{\langle \cdot,\cdot \rangle} = \|v\|_{L^2(\left[\zeta_1,\zeta_2\right])} < \infty$. Therefore, let  
		\begin{equation}
		y_1 := \frac{\eta_1}{\lambda_1 \|v\|_{\langle \cdot,\cdot \rangle}^2}v. 
		\end{equation}
		Then we have that
		\begin{align*}
		\langle a_1,y_1 \rangle &= \frac{\eta_1}{\lambda_1 \|v\|_{\langle \cdot,\cdot \rangle}^2} \langle a_1,v \rangle = \eta_1\\
		\text{and}\;\; (j\neq 1) \implies \langle y_1,a_j \rangle &= \frac{\eta_1}{\lambda_1 \|v\|_{\langle \cdot,\cdot \rangle}^2} \langle v,a_j \rangle = 0.
		\end{align*}
		This is the desired result.
	\end{proof}
	
	The following theorem provides a condition for strong duality that is, in practice, easy to check. 
	
	\begin{theorem}\label{thm:main}
		Let $[\alpha,\beta] \subset \dom f$. Let $\mathbf{S}_f$ be defined as in Lemma~\ref{lem:S}, and let $x \in \mathbf{S}$ such that $Ax=b$. Moreover, let $0 \leq \zeta_1 < \zeta_2 \leq \tau$ and $x: [\zeta_1 , \zeta_2] \rightarrow [\varepsilon_1 , \varepsilon_2] \subset \; ]\alpha,\beta [$. Let $A$ be as in \eqref{d:A} with $a_1,\dots,a_n$ linearly independent in $L^\infty([\zeta_1,\zeta_2])$ and bounded on $[\zeta_1,\zeta_2]$. Then $b \in \core A (\dom I_f)$.
	\end{theorem}
	\begin{proof}
		Let $\eta \in \mathbb{R}^n$ be a direction. Recalling Definition~\ref{def:core}, it suffices to find $t>0$ such that $b+t\eta \in A(\mathbf{S}_f) \subset A(\dom I_f)$.
		
		For $k=1,\dots,n$ let $y_k \in L^\infty([\zeta_1,\zeta_2])$ be functions bounded on $[\zeta_1,\zeta_2]$ satisfying:
		\begin{align}
		\langle y_k, a_k \rangle_{L^2([\zeta_1,\zeta_2])} &= \int_{\zeta_1}^{\zeta_2} y_k(s)a_k(s)ds =   \eta_k \label{eqn:etaequality}\\
		\text{and}\quad \langle y_k, a_j \rangle_{L^2([\zeta_1,\zeta_2])} &= 0 \quad \text{for} \; j \neq k. \label{eqn:mutualorthogonality}
		\end{align}
		We may do this by the linear independence of $a_1,\dots, a_n$ on $\left[\zeta_1,\zeta_2\right]$ (simply apply Lemma~\ref{lem:linearindependence} with $p=\infty$ and $M={\rm span}\{a_1,\dots,a_n\} \cap L^\infty (\left[\zeta_1,\zeta_2\right])$). Since $y_1,\dots,y_n$ are bounded on $[\zeta_1,\zeta_2]$, there exists $\Delta > 0$ such that: 
		\begin{equation}\label{eqn:Delta}
		\underset{s \in [\zeta_1,\zeta_2], k \in \{1,\dots,n \}}{\max}\{|y_k(s)| \} < \Delta.
		\end{equation}
		Choose $t>0$ such that 
		\begin{equation}\label{eqn:tDelta}
		t(n\Delta) < \min\{\varepsilon_1-\alpha,(\beta-\varepsilon_2) \} .
		\end{equation} 
		Finally, let
		\begin{align*}
		y: [0,\tau] &\rightarrow \;\; \left]-t(n\Delta),t(n\Delta) \right[\\
		\text{by}\quad y:s &\mapsto \begin{cases} t\sum_{j=1}^n y_j(s) & \text{if}\;\; s \in [\zeta_1,\zeta_2]\\
		0 & \text{otherwise}\end{cases}.
		\end{align*}
		We will show that:
		\begin{enumerate}
			\item \label{P1} $x+y \in \mathbf{S}_f$ \quad and
			\item \label{P2} $A(x+y) = b+t\eta$.	
		\end{enumerate}
		Conditions \ref{P1} and \ref{P2} combined demonstrate that 
		$$
		b+t\eta \in A (\mathbf{S}_f) \subset A (\dom I_f),\quad \text{and thus} \quad x \in \core (A \dom I_f).
		$$
		
		\noindent \ref{P1}: To see that $x+y \in \mathbf{S}_f$, we consider two cases.\\ 
		
		\noindent\textbf{Case 1:} If $s \notin [\zeta_1,\zeta_2]$ then $(x+y)(s)=x(s) \in [\alpha,\beta]$.\\
		
		\noindent\textbf{Case 2:} If $s \in [\zeta_1,\zeta_2]$, then we have from the assumptions that
		\begin{equation}\label{eqn:xs}
		x(s) \in \;\; \left[\varepsilon_1,\varepsilon_2 \right].
		\end{equation}
		Additionally, from \eqref{eqn:Delta} and \eqref{eqn:tDelta} we have that
		\begin{align}\label{eqn:ys}
		\resizebox{\textwidth}{!}{%
			$|y(s)|\leq t\sum_{k=1}^n |y_k(s)| \leq tn\left(\underset{s \in [\zeta_1,\zeta_2], k \in \{1,\dots,n \}}{\max}\{|y_k(s)| \} \right) \leq t(n\Delta) < \min \{\varepsilon_1-\alpha, (\beta-\varepsilon_2)  \}.$%
		}  
		\end{align}
		Now \eqref{eqn:xs} and \eqref{eqn:ys} combined show that $(x+y)(s) \in \left]\alpha,\beta \right[$. Combining the above two cases, we have that
		$$
		(\forall s \in [0,\tau]) \;\; (x+y)(s) \in [\alpha,\beta],\quad \text{and so}\quad x+y \in \mathbf{S}_f.
		$$
		\noindent \ref{P2}: To see that $A(x+y)=b+t\eta$, notice that by definition
		\begin{align}\label{eqn:Axy}
		A(x+y) = \left(\int_0^\tau a_1(s) (x(s)+y(s))ds, \dots , \int_0^\tau a_n(s) (x(s)+y(s))ds \right).
		\end{align}
		Thus the $k$th term of $A(x+y)$ is
		\begin{subequations}\label{eqn:bketak}
		\begin{align}
		\int_0^\tau a_k(s)(x(s)+y(s)) ds &= \int_0^{\tau} a_k(s)x(s)ds + \int_{\zeta_1}^{\zeta_2} a_k(s)y(s)ds \label{eqn:bketak1} \\
		&= b_k + \int_{\zeta_1}^{\zeta_2} a_k(s)y(s) ds \label{eqn:bketak2} \\
		&= b_k + \int_{\zeta_1}^{\zeta_2} a_k(s)\left( t\sum_{j=1}^n y_j(s)\right) ds \label{eqn:bketak3} \\
		&= b_k + \sum_{j=1}^n t\langle a_k,y_j \rangle_{L^2([\zeta_1,\zeta_2])} \label{eqn:bketak4} \\
		&= b_k + t\langle a_k,y_k \rangle_{L^2([\zeta_1,\zeta_2])} \label{eqn:bketak5} \\
		&= b_k + t \eta_k. \label{eqn:bketak6}
		\end{align}
		\end{subequations}
		Here \eqref{eqn:bketak1} uses the fact that $s \notin [\zeta_1,\zeta_2] \implies y(s)=0$, \eqref{eqn:bketak2} uses the definition of $Ax=b$ to replace the left integral by $b_k$, \eqref{eqn:bketak3} uses only the definition of $y$, \eqref{eqn:bketak4} is the definition of the inner product on $L^2([\zeta_1,\zeta_2])$, \eqref{eqn:bketak5} uses \eqref{eqn:mutualorthogonality}, and \eqref{eqn:bketak6} uses \eqref{eqn:etaequality}.
		
		Finally, \eqref{eqn:Axy} and \eqref{eqn:bketak} combined yield 
		$$
		A(x+y)=(b_1+t\eta_1,\dots,b_n+t\eta_n)=b+t\eta.
		$$
		This completes the result.
	\end{proof}	
	
	The following example illustrate the application of Theorem~\ref{thm:main}.

	\begin{example}[The pulse]\label{ex:pulse2}
		Consider the case of Example~\ref{ex:pulse} where $\boldx$ is the pulse. In terms of Theorem \ref{thm:main}, $\boldx \in \mathbf{S}_f$ with $\alpha=0,\beta = \infty$, $0 \leq \zeta_1 < \zeta_2 \leq \frac{1}{2}$, and $0<\varepsilon_1 = 1 = \varepsilon_2 < \infty$.
	\end{example}

	\section{Approach using the quasi-relative interior}\label{s:existingtheory}
	
	Now we will compare the approach we used in Section~\ref{s:newresults} to obtain Theorem~\ref{thm:main} against the process used in the existing theory. We have the following proposition from Borwein and Lewis' first paper on quasi-relative interiors.
	\begin{proposition}[{\cite[Proposition~2.10]{borwein1988partially}}]\label{prop:subsetqri}
		Let $X,Y$ be locally convex with $C \subset X$ convex and $A:X \rightarrow \RR^n$ continuous and linear. If $\qri C \neq \emptyset$, then $A(\qri C) = \reli(AC)$.
	\end{proposition}
	As before, we want to verify the condition that $b \in {\; \rm core}(A {\; \rm dom}F)$, so that we can apply Theorem~\ref{thm:borweinzhuduality} to guarantee strong duality. By Proposition~\ref{prop:subsetqri}, it suffices to find a feasible $x \in \qri (\dom I_f)$, because if such a feasible $x$ exists,  Proposition~\ref{prop:subsetqri} tells us that 
	\begin{equation*}
	b= Ax \in A \qri (\dom I_f) \subset \reli (A \dom I_f) \subset \core (A \dom I_f).
	\end{equation*}
	We denote by $X_+$ the positive cone in $X$. In the case where $X=L^p(\RR,\mu)$, where $(\RR,\mu)$ is the real numbers together with the usual Lebesgue measure, we simply have
	\begin{equation}\label{def:X+}
	X_{+}=X_{\geq 0}\;:=\; \{x(s)\;|\; x(s) \geq 0, \;s-a.e. \}
	\end{equation}
	It may be seen from Lemma~\ref{lem:S} and Table~\ref{tab:entropies} that $X_+$ is prototypical of the domain for many entropy functionals of interest. The following example from \cite{borwein1988partially} provides an important characterization of the quasi-relative-interiors of the domains of many entropy functionals.
	
	\begin{example}[{\cite[Examples~3.11]{borwein1988partially}}]
		Let $X=L^p(T,\mu)$ with $(T,\mu)$ a $\sigma$-finite measure space and $1 \leq p < \infty$. Then we have that
		\begin{equation}\label{Q1}\tag{Q1}
		\qri (X_+) = \left \{ x\;|\;x(s)>0,\; s-a.e. \right \}.
		\end{equation}
	\end{example}
	Strong duality may be verified by finding a feasible point $x$ (feasible in the sense that $Ax=b$) that is also in $\qri (X_+)$. However, consider the example when $f$ is the Boltzmann--Shannon entropy and $b=A\boldx$ where $\boldx$ is the pulse \eqref{def:pulse}. Here we possess a priori knowledge of a point $\boldx$ that is \emph{feasible} in the sense that $A\boldx = b$. However, \eqref{Q1} tells us that the feasible point $\boldx$ is not in $\qri (X_+)$. For many problems when strong duality holds, it may be difficult to find a feasible $x \in \qri (X_+)$ by using the characterization in \eqref{Q1}.
	
	Borwein and Lewis provided a result, \cite[Theorem~2.9]{BL2}, which we recall as Theorem~\ref{thm:BLmaintheorem}, that allows us to easily check strong duality for this particular example. Before we introduce the theorem, we first need a definition from \cite{borwein1988partially}.
	\begin{definition}[Pseudo-Haar ({\cite[Definition~7.8]{borwein1992partially}})]
		Suppose that $a_i:\left[\alpha,\beta\right] \rightarrow \RR, i=1,\dots,m$ are continuous and linearly independent on every non-null subset of $\left[\alpha,\beta\right]$. Then we say the $a_i$s are pseudo-Haar on $\left[\alpha,\beta\right]$.
	\end{definition}

	For example, the monomials $a_i: s \mapsto s^k,\;i=1,\dots,n$ are pseudo-Haar on $\left[\alpha,\beta\right]$ for any $\alpha<\beta$. This choice of $a_i$ was used in, for example, \cite{BL2018,BL2016}. Now we recall the strong duality conditions of Borwein and Lewis.
	
	\begin{theorem}[{\cite[Theorem~2.9]{BL2}}]\label{thm:BLmaintheorem}
		Suppose $(T,\mu)$ is a finite measure space and that
		\begin{varsubequations}{Q2}\label{Q2}
			\begin{align}
			&0 \leq \rho \in L^p(T)\; \text{is nonzero}\; (1 \leq p \leq \infty).\label{Q2a}\tag{Q2a} \\
			\text{and} \; &a_i \in L^q(T), i=1,\dots,n\; \text{are pseudo-Haar}\label{Q2b}\tag{Q2b}
			\end{align}
		\end{varsubequations}
		Then there exists a $y \in L^{\infty}(T)$ and $\varepsilon > 0$ with $y(s) \geq \varepsilon$ almost everywhere and $\langle \rho, a_i \rangle = \langle y, a_i \rangle$, each $i$.
	\end{theorem}
	\begin{proof}
		See \cite[Theorem~2.9]{BL2} or our proof below for Theorem~\ref{thm:2sided}, which is based thereon.
	\end{proof}

	Clearly the choice of the pulse $\boldx \in L^1(\left[0,1\right])$ satisfies \eqref{Q2a}; moreover, the condition \eqref{Q2b} holds for, say, monomials $a_i: s \mapsto s^i$ where $a_1,\dots,a_n \in L^\infty(\left[0,1\right])$. Now the condition $\langle \boldx,a_i \rangle = \langle y,a_i \rangle$ is the same as $A\boldx=Ay=b$, and the condition $y(s)\geq \varepsilon,\; s-a.e.$ assures that $y \in \reli(X_+)$ by \eqref{Q1}. Thus, the existence of this $y$ guarantees strong duality by Proposition~\ref{prop:subsetqri}. Thus the conditions \eqref{Q2} are sufficient to guarantee strong duality for the choice $\rho=\boldx$, and they are easy to check. 
	
	However, \eqref{Q2} does not tell us what to do when $X_+$ \emph{properly} contains the domain of our functional $f$. For example, when $f$ is the Fermi Dirac entropy, we have $\dom f = \left[0,1\right]$. For such a case, we can also define the sets $X_{\geq \alpha},X_{\leq \beta}$ by
	\begin{align*}
	X_{\geq \alpha}\;:=&\; \{x\;|\; x(s) \geq \alpha, \;s-a.e. \},\\
	\text{and}\quad X_{\leq \beta}\;:=&\; \{x\;|\; x(s) \leq \beta, \;s-a.e. \}.
	\end{align*}
	It is clear that $X_{\leq \beta} = -X_{+}+\beta$ where $\beta$ denotes the constant function $\beta:s \mapsto \beta$. Naturally, we have that 
	\begin{subequations}\label{suitabletranslation}
		\begin{align}
			\qri X_{\leq \beta} &= -\qri X_{+}+\beta = \left \{x \;|\; x(s) < \beta,\;s-a.e. \right \}	\\
			\text{and\;analogously}\quad\qri X_{\geq \alpha} &= -\qri X_{+}+\alpha = \left \{x \;|\; x(s) > \alpha,\;s-a.e. \right \}.
		\end{align}
	\end{subequations}
	Clearly we then have that
	\begin{align}\label{qricap}
	\qri X_{\geq \beta} \cap \qri X_{\leq \alpha} = \left \{x\;|\; \alpha < x(s) < \beta,\;s-a.e.\right \}.
	\end{align}
	This leads us to consider Theorem~\ref{thm:qri}, for which we employ another result of Borwein and Lewis.
	\begin{theorem}[{\cite[Theorem~2.13]{borwein1988partially}}]\label{thm:qri_intersections}
		Let $X$ be locally convex and $C,D \subset X$ be convex. If $C \cap \inte D \neq \emptyset$, then $(\qri C) \cap (\inte D) = \qri (C \cap D)$. 
	\end{theorem}

	\begin{theorem}\label{thm:qri}
		Let 
		\begin{equation}\label{BLcriterion}
		Ay=b\quad \text{ and}\quad \alpha < y(s) < \beta, \;\;s-a.e.,\quad\text{and}\quad\left[\alpha,\beta\right] \subset \dom f.
		\end{equation}
		Then $b \in \core (A\dom I_f)$.
	\end{theorem}
	\begin{proof}
		We have that
		\begin{subequations}\label{capequation}
		\begin{align}
		b=Ay &\in A (\qri X_{\geq \alpha} \cap \qri X_{\leq \beta}),\label{cap1} \\
		 &= A \qri X_{\geq \alpha} \cap A qri X_{\leq \beta}, \label{cap2}\\
		 &= \reli A X_{\geq \alpha} \cap  \reli A X_{\leq \beta}, \label{cap3}\\
		 &=\qri(AX_{\geq \alpha}) \cap \inte(AX_{\leq \beta}), \label{cap4}\\
		 &= \qri(AX_{\geq \alpha} \cap AX_{\leq \beta}), \label{cap5}\\
		 &= \qri(A(X_{\geq \alpha} \cap X_{\leq \beta})), \label{cap6}\\
		 &\subset \qri A \dom I_f \label{cap7}\\
		 &= \core A \dom I_f \label{cap8}.
		\end{align}
		\end{subequations}
		Here the inclusion from \eqref{cap1} is true because of \eqref{qricap} and the condition that $\alpha < y(s) < \beta$, \eqref{cap2} uses the linearity of $A$, and \eqref{cap3} is true by Proposition~\ref{prop:subsetqri}. We have that \eqref{cap4} is true, because $AX_{\geq \alpha}$ and $AX_{\leq \beta}$ are nonempty cones with dimension $n$ in $\mathbb{R}^n$, and so we have that their interiors and relative interiors coincide. We have that \eqref{cap5} is true by Theorem~\ref{thm:qri_intersections}, and \eqref{cap6} again uses the linearity of $A$. Next, \eqref{cap7} holds because $\aff A(X_{\leq \alpha} \cap X_{\leq \beta})= \aff A\dom I_f=\RR^n$ and $X_{\leq \alpha} \cap X_{\leq \beta} \subset \dom I_f$ by Lemma~\eqref{lem:S}. Finally, \eqref{cap8} is true by convexity. Altogether, \eqref{capequation} shows the desired result.
	\end{proof}
	
	The key, then, is to find a $y$ that satisfies \eqref{BLcriterion}. The existence of this $y$ does not immediately follow from Theorem~\ref{thm:BLmaintheorem}, but Theorem~\ref{thm:BLmaintheorem} may be extended by slightly modifying Borwein and Lewis' proof (from \cite[Theorem~2.9]{BL2}) as follows in Theorem~\ref{thm:2sided}.
	\begin{theorem}[A simple extension of {\cite[Theorem~2.9]{BL2}}]\label{thm:2sided}
		Suppose $(T,\mu)$ is a finite measure space and $x \in L^P(T)$ satisfies $\alpha \leq x \leq \beta$ where
		\begin{equation}\label{measure>0}
		\mu(\{t\;|\; \alpha<x(s)<\beta  \}) >0.
		\end{equation}
		Suppose further that $a_i \in L_q(T) i=1,\dots,n$ are pseudo-Haar. Then there exists a $y \in L^\infty(T)$ and $\varepsilon>0$ with $y(t)\geq \varepsilon$ almost everywhere and $\langle x, a_i \rangle = \langle y, a_i \rangle $ for each $i$. 
	\end{theorem}
\begin{proof}
	Without loss of generality, let $0 = \alpha, 1 = \beta$. Since $\mu(\{t\;|\; 0<x(s)<1  \}) >0$, there exists $T_1 \subset T$ with $\mu(T_1) > 0$ and a $\delta>0$ such that $x(t) \geq \delta$ almost everywhere on $T_1$. We claim
	\begin{equation}\label{BL2.10}
	C1:=\left \{\left(\int_{T_1}ua_i d\mu \right)_{i=1}^n  \;|\; u \in L^\infty (T_1) \right \} = \RR^n.
	\end{equation}
	Suppose to the contrary. As $C1$ is a subspace and is linearly dependent on $\RR^n$, there exists a linear function $\lambda \in \RR^n$ with $C1$ contained in its nullspace. This means that
	\begin{equation*}
	\sum_{i=1}^n \lambda_i \int_{T_1}ua_i d\mu = 0 \quad \text{for\;all}\;\; u \in L^\infty (T_1).
	\end{equation*}
	Thus we have that $\sum_{i=1}^n \lambda_i a_i(t) = 0$ almost everywhere on $T_1$. This is a contradiction, because the $a_i$ are pseudo-Haar, and so they are linearly independent on $T_1$. Thus \eqref{BL2.10} holds. Now define
	\begin{equation*}
	C2 = \left \{ \left(\int_{T_1} ua_i d\mu \right)_{i=1}^n \;|\; u \in L^\infty(T_1),\;\; \|u\|_\infty < \delta/2    \right \}.
	\end{equation*}
	Since, by \eqref{BL2.10}, $\cone C2= \RR^n$, we have that $0 \in \inte C2$ by \cite[Corollary~6.4.1]{Roc70}. Define the sequence $(x_m)_{m=3}^\infty \subset L^\infty (T) $ by 
	\begin{equation*}
	x_m(t) = \begin{cases}
	1-\frac{1}{m} &\text{if}\;\; x(t)>1-\frac{1}{m};\\
	x(t) & \text{if}\;\; \frac{1}{m} \leq x(t) \leq 1-\frac{1}{m};\\
	\frac{1}{m} & \text{if}\;\; x(t)< \frac{1}{m}.
	\end{cases}
	\end{equation*}
	For $p< \infty$ we have 
	\begin{align*}
	\|x_m-x\|_p^p &\leq \int_{\{t\;|\;x(t)>1-\frac{1}{m} \}} \left(1-\left(1-\frac{1}{m}\right)\right)^p d\mu+ \int_{\{t\;|\; x(t)<\frac{1}{m} \}}\left(\frac{1}{m}\right)^p d\mu\\
	&= \int_{\{t\;|\;  \leq x(t) \notin \left[1/m,1-1/m \right]  \}} \left(\frac{1}{m} \right)^p d\mu\\
	&\rightarrow  0 \;\;\text{as}\;\; m \rightarrow \infty.
	\end{align*}
	Otherwise, in the $p=\infty$ case, we have $\|x_m-x\|_\infty \leq 1/m$ for all $m$. Thus we have that $(\langle x_m - x, a_i\rangle)_{i=1}^m \rightarrow 0$ as $m \rightarrow \infty$, and so, since $0 \in \inte C2$, we have that for $m>1/\delta$ sufficiently large, $(\langle x_m -x, a_i \rangle)_{i=1}^n \in C2$. Thus we may find $v \in L_\infty(T)$ such that $\langle x_m-x,a_i \rangle = \langle v,a_i \rangle$ for $i=1,\dots,n$ and $\|v\|_\infty < \delta/2$ with $v(t)=0$ almost everywhere on $T_1^\complement$. Set $y=x_m -v$. Then we have $y\in L^\infty (T)$ and $y(t)=x_m(t)\in \left[1/m,1-1/m\right]$ almost everywhere on $T_1^\complement$. On $T_1$ we have $1-\frac{1}{m}\geq 1-\delta \geq x(t) \geq \delta >1/m$ almost everywhere, so $x_m(t) \in \left[\delta,1-\delta \right]$ almost everywhere. Since $v(t) \leq \delta/2$ almost everywhere, we have that $y(t) \in \left[\delta/2,1-\delta/2 \right]$ almost everywhere on $T_1$. Altogether we have that $y \in \left[ \min \{\delta/2,1/m \},1-\min \{\delta/2,1/m \} \right]$ almost everywhere on $T$. Finally, since $y=x_m-v$ and $\langle x_m-x,a_i \rangle = \langle v,a_i \rangle$ for $i=1,\dots,n$ we have that
	$$
	\text{for}\;i=1,\dots,n,\;\;\langle y,a_i \rangle = \langle x_m -v , a_i \rangle = \langle x_m,a_i \rangle + \langle v,a_i \rangle = \langle x_m,a_i \rangle - \langle x_m-x,a_i \rangle = \langle x,a_i \rangle.
	$$
	This concludes the result.	
\end{proof}

While Theorem~\ref{thm:2sided} is a straightforward extension of \cite[Theorem~2.9]{BL2}, it is not actually proven therein. Borwein and Lewis simply point out \cite[Example~5.6(vi)]{BL2} that, in such cases, the problem reduces to finding a $y$ that satisfies \eqref{BLcriterion}. In their later work \cite{borwein1993partially}, Borwein and Lewis list the extension from Theorem~\ref{thm:2sided} as the primal constraint qualification $(PCQ_3)$, and they attribute the first proof to \cite{lewispseudo}, which may be an early reference to Lewis' detailed exposition on the consistency of moment systems \cite{lewis1995consistency}.

To conclude our discussion of how the quasi-relative interior approach here relates to our approach to developing the constraint qualification by using the definition of the core, we have the following proposition that shows how \eqref{measure>0} implies the existence of our interval $\left[ \zeta_1,\zeta_2 \right]$ from Theorem~\ref{thm:main}.

	\begin{proposition}\label{prop:sigma_algebra}
	Let $\mu \left( \left \{x \;|\; \alpha<f(x)<\beta \right \} \right) >0$. Then there exists $\epsilon,\delta >0$ such that 
	\begin{equation}\label{eqn:contradiction}
	\mu \left( \left \{x \;|\;  \alpha+\epsilon \leq f(x)\leq \beta-\epsilon \right \} \right) = \delta > 0.
	\end{equation} 
\end{proposition}
\begin{proof}
	Suppose for a contradiction that such a $(\epsilon,\delta)$ pair does not exist. Let
	\begin{alignat*}{2}
	&& \epsilon_n :=& \frac{\beta-\alpha}{2^{n}}\quad n \in \mathbb{N};\\
	\text{}\quad&& A_n:=& \left \{x \in T\;|\; \alpha+\epsilon_n \leq f(x) \leq \beta-\epsilon_n  \right \};\\
	\text{and}\quad && B_n:=& A_{n+1}\setminus \cup_{j=1}^n A_{j}, \quad n \in \mathbb{N}.
	\end{alignat*}
	Since we have supposed that \eqref{eqn:contradiction} does not hold for any $(\epsilon,\delta) \in \mathbb{R}_{++}^2$, we have that $\mu(A_n) = 0$ for all $n$, and so $\mu(B_n) = 0$ for all $n$. Since $(B_n)_{n \in \mathbb{N}}$ is a disjoint, countable family of members of a $\sigma$-algebra,
	\begin{equation}\label{measurespace1}
	\mu \left(\bigcup_{n=1}^\infty B_n \right) = \sum_{n=1}^\infty \mu (B_n) = 0.
	\end{equation}
	See, for example, \cite{Brezis}. However, we also have that
	\begin{equation}\label{measurespace2}
	\bigcup_{n=1}^\infty B_n = \left \{x \;|\;  \alpha \leq f(x)\leq \beta \right \}.
	\end{equation}
	Together, \eqref{measurespace1} and \eqref{measurespace2} force
	$$
	\mu \left( \left \{x \;|\; \alpha<f(x)<\beta \right \} \right) =0,
	$$
	which is a contradiction.
\end{proof}
	From Proposition~\ref{prop:sigma_algebra}, it is clear why the constraint qualification conditions of Theorem~\ref{thm:2sided} are more restrictive than those of Theorem~\ref{thm:main} for the problem \eqref{Iprimalproblem}. We only require that $a_1,\dots,a_n$ be linearly independent on one set: $\left[ \zeta_1,\zeta_2 \right]$. 
	We gain different insights through the two different approaches to the problem. By working in the core, we obtained a sufficiency condition that is no more restrictive in practice than the state-of-the-art in the literature. Moreover, we obtain it with very little scaffolding. 

\section{A Computed Example}\label{s:example}

Consider the pulse from Examples~\ref{ex:pulse} and \ref{ex:pulse2}. Theorem~\ref{thm:main} admits the use of new moment functions for solving \eqref{Iprimalproblem}, provided that the moment functions are linearly independent on some nonempty $\left[\zeta_1,\zeta_2\right] \subset \left[0,1/2 \right]$. In particular, let
\begin{equation}\label{piecewiseaj}
a_i(t) = \begin{cases}
t^{i-1} & \text{if}\;\; t \leq 1/2;\\
1 & \text{if}\;\; t \geq 1/2,
\end{cases}\quad \text{for}\;\; j=1,\dots,n.
\end{equation}

We compute with $f$ as the negative translated Boltzmann--Shannon entropy, and $\tau = 1$. In this case, $\left[\zeta_1,\zeta_2 \right]$ is any subset of $\left[0,1/2\right]$. In Figure~\ref{fig:computed_example}, we show the solution obtained by computing with piecewise-defined $a_i$ in \eqref{piecewiseaj}. For comparison, we also show the solution obtained when the $a_i:t \rightarrow t^{i-1}$ are monomials on the entire interval $\left[0,1\right]$. The latter are the moments used in \cite{BL2018,BL2016,ScottPhD}, to which the reader is referred for more details about computation. At left, we see that the use of the piecewise $a_i$ apparently tames the observed Gibbs phenomenon near $1/2$. At right, we rescale for clearer comparison.

\begin{figure}
	\begin{center}
	\adjustbox{trim={.1\width} {.43\height} {.2\width} {0.08\height},clip}
	{\includegraphics[width=.65\textwidth]{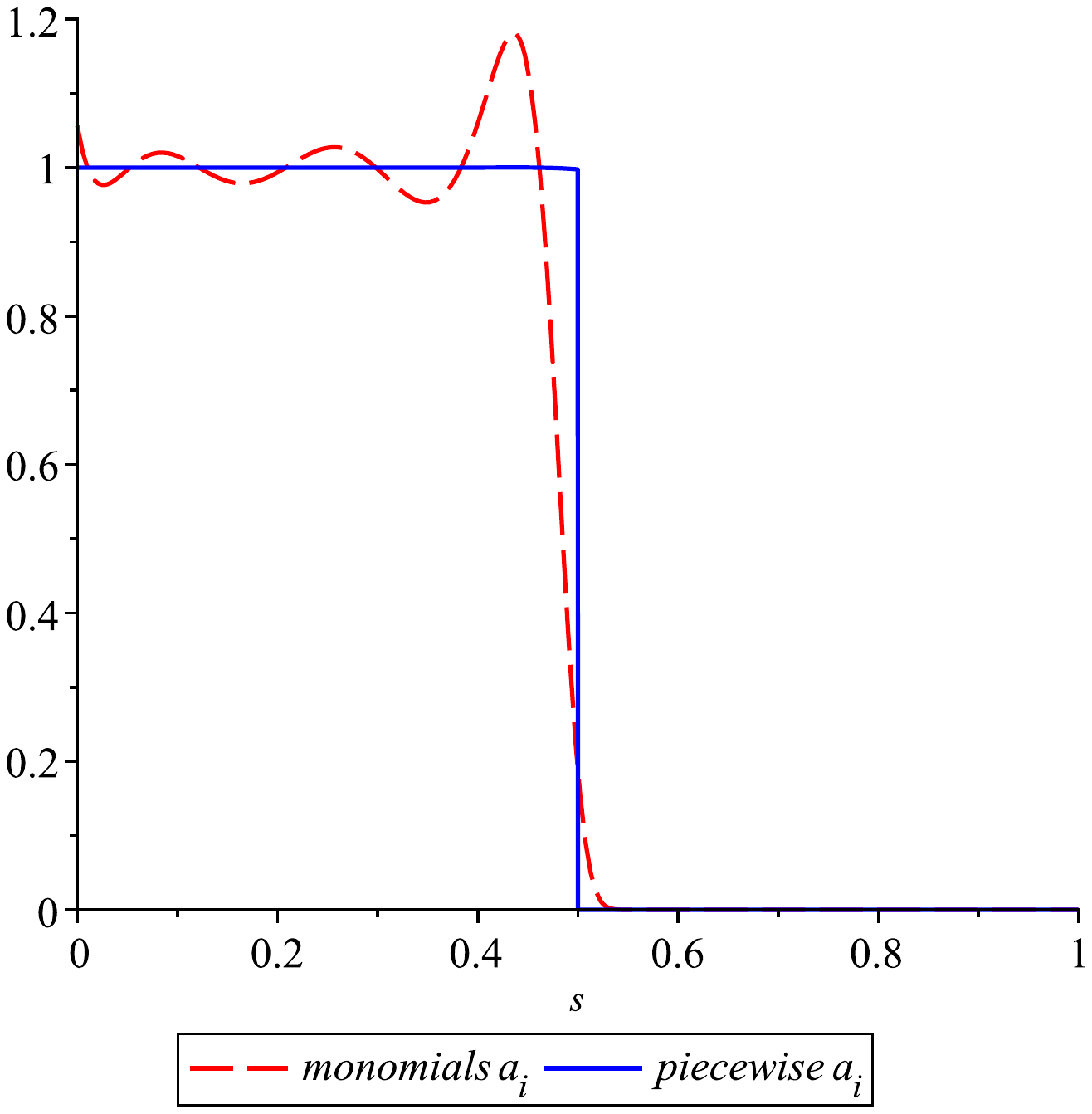}}	
	\adjustbox{trim={.1\width} {.43\height} {.2\width} {0.08\height},clip}
	{\includegraphics[width=.65\textwidth]{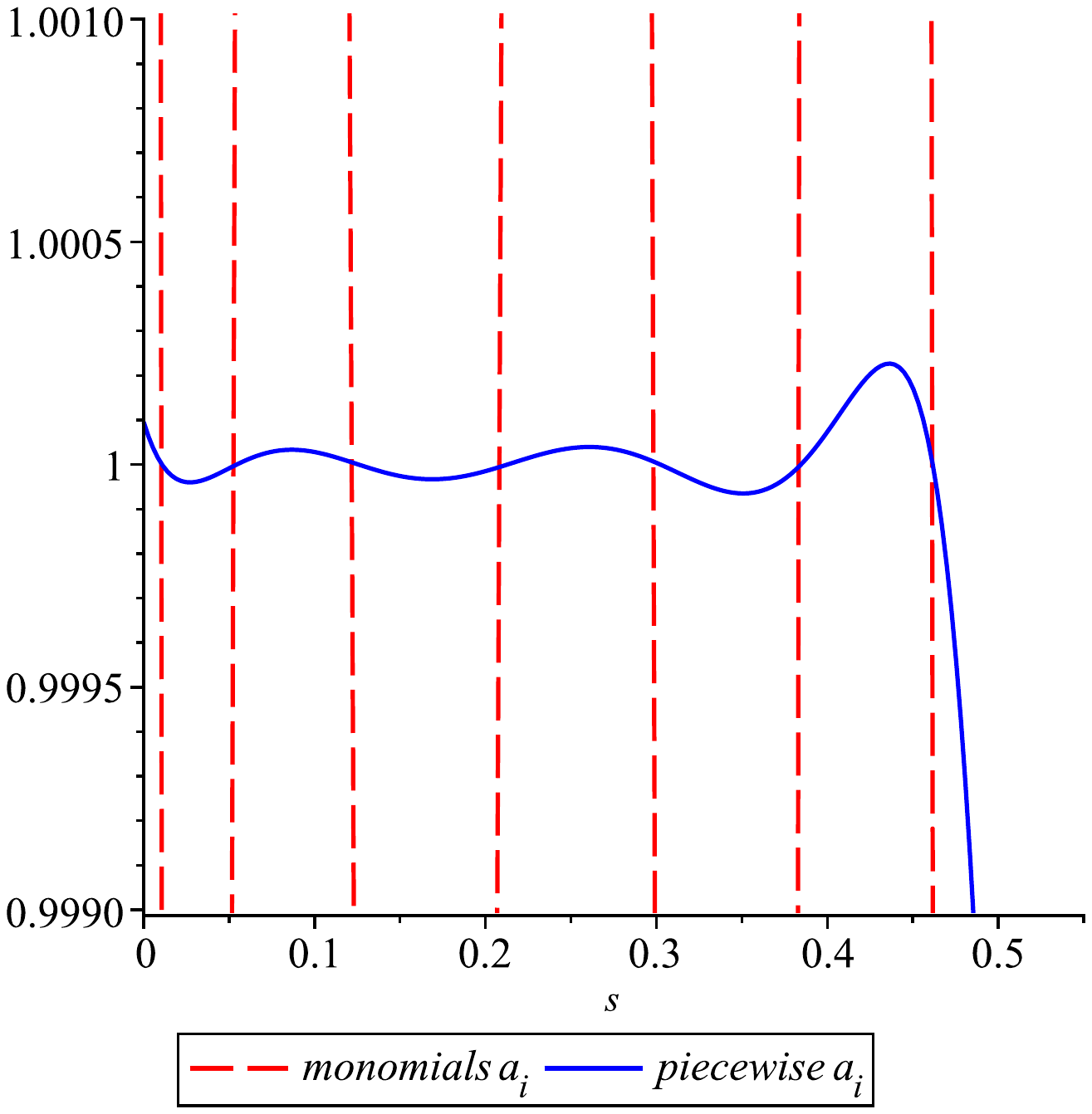}}		
	\end{center}
\caption{The computed example from Section~\ref{s:example}.}\label{fig:computed_example}
\end{figure}

\section{Conclusion}\label{s:conclusion}

Let us compare and contrast the two approaches to finding easy-to-check conditions for strong duality. In Section~\ref{s:newresults}, we started with a direction vector $\eta \in \RR^n$ and an $x$ that satisfied 
\begin{equation}\label{xfeasible}
x: [\zeta_1 , \zeta_2] \rightarrow [\varepsilon_1 , \varepsilon_2] \subset \; ]\alpha,\beta [ \subset \dom I_f,
\end{equation}
and then we showed how we may construct a function $x+y$ and $t>0$ that satisfies $x+y \in \dom f$ and  $A(x+y)=b+t\eta$. This shows that $b \in \core A \dom I_f$. 

In Section~\ref{s:existingtheory}, we started with a function that satisfies \eqref{eqn:contradiction} and so also satisfies \eqref{xfeasible} by Proposition~\ref{prop:sigma_algebra}. We then show how to construct another function $y$ that satisfies $Ay=b$ with $y \in \qri \dom I_f$, which guarantees $b \in \core A \dom I_f$ by Theorem~\ref{thm:qri}. 

The proof in Section~\ref{s:newresults} has two main advantages. The first is that it provides an interiority condition that is easy to understand from a geometric standpoint. We simply look for the interval $\left[ \zeta_1,\zeta_2 \right]$ and then can safely add a function $y$ to the function $x$ on this interval while keeping $x+y \in \dom I_f$. The second advantage is that this geometric construction makes clear why we only need the $a_i$ to be linearly independent on $\left[\zeta_1,\zeta_2 \right]$.

The proof in Section~\ref{s:existingtheory} has a significant theoretical advantage of its own, which is that it shows how one can construct a function $y$ that lives in $\qri \dom I_f$. The two approaches are related by Proposition~\ref{prop:sigma_algebra}, which shows how we may find the set $T_1$ by simply taking $T_1:=\left[\zeta_1,\zeta_2\right]$, the existence of the latter being guaranteed. Revisiting the proof of Theorem~\ref{thm:2sided}, we realize that we have only used the pseudo-Haar assumption to guarantee linear independence of the $a_i$ on the set $T_1 = \left[\zeta_1,\zeta_2\right]$, and so we only ever needed the linear independence on $T_1$ to begin with.

The study of which properties of relative interiors in finite dimensions will extend to quasi-relative interiors in infinite dimensions is an active area of current research. This study of the entropy minimization problem illustrates the insights that are to be gained by examining problems that are partially finite, and approaching them from both their finite-dimensional and infinite-dimensional sides.

	\bibliographystyle{plain}
	\bibliography{bibliography}

\end{document}